\documentclass{amsart}

\usepackage{eufrak}

\usepackage{bbold}

\usepackage{amssymb,amsmath, amscd, mathrsfs, yfonts, bbm, graphics, dsfont}

\title[On the asymptotic Boolean independence] {An asymptotic property of large matrices with identically distributed Boolean independent entries}

\author{Mihai Popa and Zhiwei Hao}

\address{ Department of Mathematics, University of Texas at San Antonio, One UTSA Circle San Antonio, Texas 78249, USA, and
}

\address{
"Simion Stoilow" Institute of Mathematics of the Romanian Academy, P.O. Box 1-764,
Bucharest, RO-70700, Romania}
\email{mihai.popa@utsa.edu}

\address{School of Mathematics and Statistics, Central South University, Changsha
410083, China, and}

\address{ Department of Mathematics, University of Texas at San Antonio, One UTSA Circle San Antonio, Texas 78249, USA
}
\email{zw\underline{ }hao8@163.com}


\thanks{This work was partially supported by the NNSFC Cooperation Grant No. 201506370093}

\newtheorem{claim}{}[section]
\newtheorem{defn}[claim]{Definition}
\newtheorem{thm}[claim]{Theorem}
\newtheorem{lemma}[claim]{Lemma}

\newtheorem{prop}[claim]{Proposition}

\newtheorem{conseq}[claim]{Consequence}

\newcommand{\tr}{\mathrm{tr}}

\newcommand{\oxi}{\overrightarrow{\xi}}

\newcommand{\oi}{\overrightarrow{i}}

\newcommand{\oik}{\overrightarrow{k}}

\newcommand{\oia}{\overrightarrow{\alpha}}

\newcommand{\alt}{\mathrm{Alt}}

\newcommand{\cA}{\mathcal{A}}

\usepackage{amssymb}
\input xy
\xyoption{all}

\usepackage{graphicx}

\begin{document}

\begin{abstract}
Motivated by the recent work on asymptotic independence relations for random matrices with non-commutative entries, we
investigate the limit distribution and independence relations for large matrices with identically distributed and Boolean independent entries.
 More precisely, we show that, under some moment conditions, such random matrices are  asymptotically $ B $-diagonal and Boolean independent from each other. The paper also gives a combinatorial condition under which such matrices are asymptotically Boolean independent from the matrix obtained by permuting the entries (thus extending a recent result in Boolean probability). 
 In particular, we show that random matrices considered are asymptotically Boolean independent from their partial transposes. 
  The main results of the paper are based on  combinatorial techniques.
\end{abstract}

\maketitle

\section{Introduction}

 The paper presents several results concerning the asymptotic behavior (dimensions  tends to infinity)  of random matrices with boolean independent and identically distributed entries. As described below, similar settings (random matrices with identically distributed entries) have been considered in literature in both classical (commutative) and free independence frameworks. Some of the results presented here (mostly the third section of the paper) can be seen as their Boolean analogues.

In the 1950's, E. Wigner published several influential results (see \cite{w1}, \cite{w2}), regarding Gaussian random matrices, i.e. self-adjoint random matrices
with their upper-diagonal entries forming an independent Gaussian family. In particular, under certain limiting conditions, such matrices converge in distribution to a semicircular variable. In 1991, D.-V. Voiculescu (see \cite{v1}), connected the free probability theory with random matrices; one of his results is that two Gaussian random matrices with independent entries are asymptotically free. This result was later improved (see, for example \cite{mingo-speicher-gaussian}, \cite{mp-orth}), in particular by K. Dykema in \cite{dykema}, where the asymptotic freeness is studied in a more general framework, Gaussian entries being replaced be identically distributed and independent random variables. O. Ryan, in \cite{ryan}, further generalized Dykema's results, describing the asymptotic distribution and asymptotic independence relations for an even larger class of random matrices, including also non-self-adjoint cases. Also, in \cite{ryan}, O. Ryan proves similar results for  matrices with 
non-commutative entries which are identically distributed and free independent.

  As described in \cite{schurmann}, classical and free independence are not the only `universal' independence relations: the literature also considers the Boolean independence (when symmetry is still assumed, but not unitality) and monotone independence (when both symmetry and unitality are not assumed). Boolean probability theory
 have been in the literature at least since early 1970's (see \cite{wanderfelds}) with various
developments, from stochastic differential equations to measure theory \cite{s-w}.
The topic has attracted an increasing interest in the recent years, such as the works  Popa and
Vinnikov \cite{popa-v} 
Gu and Skoufranis \cite{gu-s}, Jiao and Popa \cite{jp}, Liu \cite{liu1, liu2} and Popa and Hao \cite{popa-h}. This is the motivation for the present paper, which studies the asymptotic behavior of random matrices with independent identically distributed entries in the framework of Boolean probability.

The outline of the paper is as follows. In the next section (Preliminaries), we present some notations and preliminary results in Boolean probability.  Most relevant, we introduce the notion of \emph{$B$-diagonal} non-commutative variables, as a Boolean analogue of the $ R $-diagonal variables from free probability theory (see \cite{ns-rdiag}, \cite{ns}). 
Section 3 presents results concerning non-self-adjoint random matrices with identically distributed and Boolean independent entries. The main results are Theorem \ref{thm:4.1}, which states that the asymptotic distribution of such random matrix is a $B$-diagonal, and Theorem \ref{thm:4.2}, which states that matrices as above with Boolean independent entries are asymptotically Boolean independent. Thus, Section 3 can be seen as a Boolean analogue of the results presented by O. Ryan in \cite{ryan}. 
Section 4 presents results on the asymptotic independence between a random matrix with identically distributed Boolean independent entries and the matrix obtained by permuting the entries. Here we improve the results from \cite{popa-h}, giving a general sufficient condition that implies asymptotic Boolean independence between the two matrices from above. In particular, we show the asymptotic Boolean independence from matrix transposes (result somehow analogous to \cite{mp2}) and classes of matrix partial transposes (see also \cite{mp3}).
The last part of the paper, Section 5, presents a result concerning the asymptotic distribution of self-adjoint
random matrices with identically distributed and Boolean independent entries.


\section{Preliminaries}


\subsection{The lattice of interval partitions}\label{partition} ${}$


For a positive integer $n$ we shall denote by $ [ n ] $ the ordered set
$ \{1, 2, . . . , n\}$. By an \emph{interval partition} on $ [ n ] $  we shall understand a collection of disjoint subsets $ B_1,B_2,\dots,B_r $ of $ [ n ] $, called the blocks of $ \sigma $, such that there exist some positive integers
$ 0 = l(0) < l(1) < \dots l(r) = n $ with the property that
 $ B_t = [ l(t) ] \setminus [ l(t -1)] $
 for each $ t \in [ r ] $.
 In this case,we will denote
  $ \sigma = [ l(1),l(2),\dots,l(r-1),n]$.

  If each block of $\pi $ has exactly 2 elements, then $ \pi $ will be said to be an \emph{interval pairing}. We will denote the set of all interval partitions,
respectively pairings of $[n]$ by $\mathcal{I}(n)$, respectively $\mathcal{I}_2(n)$ (if $n$ is odd, then
$\mathcal{I}_2(n) = \varnothing$;
 if $n$ is even then $\mathcal{I}_2(n)$ has only one element, the partition of
blocks $\{(2k-1, 2k): 1 \leq k \leq n/2\}$).

 The set $ \mathcal{I}(n)$ has a lattice structure with respect to the partial order relation $`` \leq '' $ given by $ \pi \leq \sigma $ whenever each block of $ \pi $ is contained in one of the blocks of $ \sigma $.
 If
 $\sigma = [ l(1), l(2), \dots, l(r)]$
 and
 $ \omega = [ u(1), u(2), \dots, u(t) ] $
 are two elements of
 $ \mathcal{I}(n) $
 (in particular, $ l(r) = u(t) = n  $),
 then their \emph{meet}
 $ \sigma \wedge \omega = [ v(1), v(2), \dots, v(p)] $,
 respectively their \emph{join}
 $ \sigma \vee \omega = [ w(1), w(2), \dots, w(q)] $
 can be described as follows.
 Put
 $ v(0) = w(0) = 0 $,
 and, inductively,
 \begin{align*}
 v(k+1) = & \inf \{ l(s), u(s):\ l(s), u(s) > v(k) \}\\
 w(k+1) = & \inf \{ l(s):\  l(s) > v(k) \textrm{ and } l(s) = u(s^\prime)
 \textrm{ for some } s^\prime \in [ t ] \}.
 \end{align*}
  The maximal, respectively minimal elements are the partition with a single block, denoted by $\mathbb{1}_n$, respectively the partition with $ n $ singleton blocks, denoted by $ \mathbb{0}_n $.

If
$ \sigma = [ l(1), l(2), \dots, l(r) ] $
and
$ \pi = [ p(1), p(2), \dots, p(s) ] $
are elements from $ \mathcal{I}(n) $, respectively $ \mathcal{I}(m) $, we define their \emph{juxtaposition} $ \sigma \oplus \pi $ as the element of
$ \mathcal{I}(n + m ) $
given by
$ \sigma \oplus \pi = [l(1), l(2), \dots, l(r), p(1)+n, p(2)+ n, \dots, p(s) + n ] $.

   We refer to \cite{s-w}  and \cite{ns} for more details on the lattice structure of interval partitions.


\subsection{Boolean cumulants, Boolean independence and Bernoulli distributed variables}\label{Boolean:cumulants}


Throughout the paper, by a non-commutative probability space we will understand a pair
 $(\mathcal{A}, \varphi)$,
  where $ \mathcal A $ is a complex, unital algebra and
  $ \varphi:\mathcal{A} \longrightarrow \mathbb{C} $
   is a unital, $ \mathbb{C}$-linear  map. The elements of $ \mathcal{A} $ will be called (noncommutative) variables. Although we will require a $ \ast$-algebra
   structure on $ \mathcal{A} $, no positivity properties of $ \varphi $ will be used nor required.

 For $ r $ a positive integer, we define, following \cite{s-w}, the $r$-th Boolean cumulant associated to $\varphi$ as the multilinear  complex map from $ \mathcal{A}^r $  given by the recurrences:
 \begin{equation} \label{BC1}
 \varphi(a_1 a_2 \dots a_n) =
 \sum_{\pi\in \mathcal{I}(n)}
 \mathfrak{b}_\pi \big[  a_1, a_2, \dots, a_n  \big],
 \end{equation}
 where
 $\displaystyle
 \mathfrak{b}_\pi \big[  a_1, a_2, \dots, a_n  \big]
 =
 \prod_{\substack{V\in \pi\\V= \{ l+1, l+2 , \dots, l+k \} }}
 \mathfrak{b}_k(a_{l +1}, a_{l+2},\dots, a_{l+k}).
 $

\begin{defn} Let $(\mathcal{A}, \varphi)$ be a non-commutative probability space. A family of (not
necessarily unital) subalgebras $\{ \mathcal{A}_j\}_{ 1 \leq j \leq n} $ of $ \mathcal{A} $ is said to be Boolean independent
with respect to $ \varphi$ if
\[\varphi(a_1a_2\cdots a_m)=\varphi(a_1)\varphi(a_2)\cdots\varphi(a_m),\]
for any $m\geq 1$ whenever
$a_j\in \mathcal{A}_{i(j)}$
 with $ i(j) \in [n] $ and
  $i(1) \neq i(2)\neq\cdots \neq i(m)$.
\end{defn}
  A set of random variables
$\{ a_j\}_{ 1 \leq j \leq n} \in \mathcal{A}$
is said to be Boolean independent if the family of non-unital
subalgebras $ \mathcal{A}_j $, which are generated by $a_j$ is boolean independent.
 An equivalent condition (see \cite{s-w}) for  Boolean independence is
 \[
 \mathfrak{b}_m(a_1,a_2,\cdot\cdot\cdot, a_m)=0
  \]
 whenever
  $a_j \in\mathcal{A}_{i(j)}$
   such that not all $i(j)$ are equal.

   The central limit distributions corresponding to Boolean independence are the
   Bernoulli distributions, described bellow.
\begin{defn}
A self-ajoint element $ X $ of $ \mathcal{A} $ is said to be \emph{Bernoulli} distributed of mean 0 and variance $ \alpha > 0 $ with respect to $ \varphi $ if
\[
\varphi(X^n) = \left\{
\begin{array}{ll}
0 & \textrm{ if  $ n $ is odd}\\
\alpha^{ \frac{n}{2} } & \textrm{ if $ n $ is even },
\end{array}  \right.
\]
or, equivalently, if
 $ \displaystyle \mathfrak{b}_n ( X, X, \dots, X) = \delta_{n}^2 \cdot \alpha. $
\end{defn}

\subsection{ $ B $-diagonal   variables}

 Following the definition of $ R $-diagonal elements (see \cite{ns}, Definition 15.3; also \cite{ns-rdiag}), we will define their Boolean analogues. First,  suppose that
 $ X $ is a non-selfadjoint random variable in $ \mathcal{A} $,
 and that
  $ \xi_1, \xi_2, \dots, \xi_n $
  are either $ 1 $ or $ \ast$.
  A Boolean cumulant
  $ \mathfrak{b}_n (X^{\xi_1}, X^{\xi_2}, \dots, X^{\xi_n}) $
  is said to be \emph{alternating} if $ n $ is even and
  $ \xi_k \neq \xi_{ k+1} $
   for all
   $ k \in [n-1]$.
   \begin{defn} Let $ ( \mathcal{A}, \varphi) $ be a non-commutative probability space
   such that $\mathcal{A} $ is a $ \ast$-algebra.

\noindent\emph{(i)} A non-selfadjoint random variable $ X $ from  $ {A} $ is said to be \emph{$ B $-diagonal} if all its non-alternating Boolean cumulants cancel.

\noindent\emph{(ii)} If $ X  \in \mathcal{A} $ is $ B $-diagonal, the sequences
$(\alpha_n)_{n} $
and
 $ (\beta_n)_n $
 where
 \begin{align*}
 \alpha_n = &\mathfrak{b}_{2n} ( X, X^\ast, \dots, X, X^\ast )\\
 \beta_n = & \mathfrak{b}_{2n} (X^\ast, X, \dots, X^\ast, X)
 \end{align*}
 are called the \emph{determining sequences} of $ X $.
 \end{defn}

   Next, we shall present two immediate properties of $B $-diagonal variables (comparable to \cite{ns}, Proposition 15.8 and Corollary 15.11).

 \begin{prop}\label{prop:B}
 Let $ (\mathcal{A}, \varphi)$ be a non-commutative probability space such that $ \mathcal{A} $ is a $\ast$-algebra.

\noindent \emph{(i)} If $ X $ is a $ B $-diagonal variable, then $ X^\ast X $ and $XX^\ast$
 are Boolean independent.

\noindent \emph{(ii)} If $ a, b \in \mathcal{A} $ are such that
$ \varphi(a) = \varphi(a^\ast) = 0 $
and $ \{ a, a^\ast\} $, $\{ b, b^\ast\} $
are Boolean independent, then, for $ Y = ab $
and $ \xi_1, \dots, \xi_n \in \{1, \ast\} $,
\[ \mathfrak{b}_n(Y^{\xi_1}, Y^{\xi_2}, \dots, Y^{\xi_n}) =0 \]
unless $ n $ is even and $ \xi_{2j-1} = \ast $, $ \xi_{2j} = 1 $ for $ j = 1, 2, \dots, n \slash 2 $.
In particular, $Y $ is $ B $-diagonal.
 \end{prop}

 For the proof of Proposition \ref{prop:B}, we need the following Lemma, proved in \cite{popa09}.

 \begin{lemma}\label{mp09}
 Suppose that $ \mathcal{A}_1 $ and $ \mathcal{A}_2 $ are two Boolean independent subalgebras in a non-commutative probability space, that $ x \in \mathcal{A}_1$,
 $ y \in \mathcal{A}_2 $ and that $ a_1 , a_2 $ are elements in the non-unital algebra generated by $ \mathcal{A}_1 $ and $ \mathcal{A}_2 $. Then
 \[ \varphi( a_1  x  y  a_2) = \varphi(a_1  x ) \cdot  \varphi(y  a_2).\]
 \end{lemma}

Let us proceed now to proving Proposition \ref{prop:B}.

\begin{proof}
 For part (i), it suffices to show that
  \begin{equation}\label{p}
  \varphi \big( (X X^\ast)^p X^\ast X^{\xi_1} X^{\xi_2} \cdots X^{\xi_m}   \big) =
  \varphi\big((X X^\ast)^p \big) \cdot
 \varphi\big( X^\ast X^{\xi_1} X^{\xi_2} \cdots X^{\xi_m}   \big)
  \end{equation}
  for all positive integers $p $ and $ m $ and all
  $ \xi_1, \xi_2, \dots, \xi_m \in \{ 1, \ast\} $.
  The recurrence (\ref{BC1}) gives that
 \[ \varphi \big( (X X^\ast)^p X^\ast X^{\xi_1} X^{\xi_2} \cdots X^{\xi_m}   \big)
   =
   \sum_{ \sigma \in \mathcal{I}(n)}
  \mathfrak{b}_{\sigma} \big[\underbrace{X, X^\ast, \dots, X, X^\ast}_{2p}, X^\ast, X^{\xi_1}, \dots, X^{\xi_m}  \big]
  \]
 Let
  $ n = 2p + m + 1 $
 and
 $ \pi = [ 2p, n ] \in \mathcal{I}(n) $.
 If $ \sigma $ is an element of $ \mathcal{I}(n) $ such that
 $ \sigma \nleq \pi $
 then $ 2p $ and $ 2p + 1 $ are in the same block of $ \sigma $. The Boolean cumulant corresponding to that block cancels, since $ X $ is $ B $-diagonal, hence the term corresponding to $ \sigma $ cancels in the summation above. Thus
  \[ \varphi \big( (X X^\ast)^p X^\ast X^{\xi_1} X^{\xi_2} \cdots X^{\xi_m}   \big)
    =
    \sum_{ \substack{ \sigma \in \mathcal{I}(n)\\ \sigma \leq \pi }}
   \mathfrak{b}_{\sigma} \big[\underbrace{X, X^\ast, \dots, X, X^\ast}_{2p}, X^\ast, X^{\xi_1}, \dots, X^{\xi_m}  \big].
   \]
   But
    $ \pi = \mathbb{1}_{2p} \oplus \mathbb{1}_{m+1} $,
     hence
   $ \sigma \leq \pi $
   is equivalent to
   $ \sigma = \sigma_1 \oplus \sigma_2 $
   for some
   $ \sigma_1 \in \mathcal{I}(2p) $,  $ \sigma_2 \in \mathcal{I}(m+1) $,
   therefore
   \begin{align*}
 \varphi \big( (X X^\ast)^p &  X^\ast  X^{\xi_1} X^{\xi_2} \cdots X^{\xi_m}   \big)\\
    = &
  \sum_{ \substack{ \sigma_1 \in \mathcal{I}(2p)\\ \sigma_2 \in \mathcal{I}(m+1) }}
   \mathfrak{b}_{\sigma_1 \oplus \sigma_2} \big[\underbrace{X, X^\ast, \dots, X, X^\ast}_{2p}, X^\ast, X^{\xi_1}, \dots, X^{\xi_m}  \big]\\
    = &
\sum_{ \substack{ \sigma_1 \in \mathcal{I}(2p)\\ \sigma_2 \in \mathcal{I}(m+1) }}
( \mathfrak{b}_{\sigma_1} \big[ \underbrace{X, X^\ast, \dots, X, X^\ast}_{2p} \big]
  \cdot
  \mathfrak{b}_{\sigma_2}
  \big[ X^\ast, X^{\xi_1}, \dots, X^{\xi_m}\big]  )\\
  = & \big(
  \sum_{ \sigma_1 \in \mathcal{I}(2p) }
   \mathfrak{b}_{\sigma_1} \big[ \underbrace{X, X^\ast, \dots, X, X^\ast}_{2p} \big]
  \big)
  \cdot\big(
  \sum_{\sigma_2\in \mathcal{I}(m+1)}
  \mathfrak{b}_{\sigma_2}
  \big[ X^\ast, X^{\xi_1}, \dots, X^{\xi_m}\big]
  \big)
   \end{align*}
   and (\ref{p}) follows applying again the recurrence (\ref{BC1}).

   For part (ii), using recurrence (\ref{BC1}), it suffices to show that
   \[
   \varphi \big( (ab)^{\xi_1} (ab)^{\xi_2} \cdots (ab)^{\xi_n} \big) = 0
   \]
   unless  $ n $ is even and
    $ \xi_{2j-1} = \ast $, $ \xi_{2j} = 1 $ for
     $ \displaystyle j \in [  n \slash {2} ] $.

   If $ \xi_1  = 1 $, then Lemma \ref{mp09} gives that
   \[
  \varphi \big( (ab)^{\xi_1} (ab)^{\xi_2} \cdots (ab)^{\xi_n} \big)
  = \varphi(a)
 \varphi\big( b (ab)^{\xi_2} \cdots (ab)^{\xi_n} \big) = 0.
  \]
  Similarly, if $ \xi_n = \ast $,
  \[
  \varphi \big( (ab)^{\xi_1} (ab)^{\xi_2} \cdots (ab)^{\xi_n} \big)
  =
  \varphi \big( (ab)^{\xi_1} (ab)^{\xi_2} \cdots (ab)^{\xi_{n -1} b^\ast } \big) \cdot \varphi(a^\ast) = 0.
  \]
 If $ \xi_k = \xi_{ k + 1} =1 $,
  then
  $ (ab)^{\xi_k}(ab)^{ \xi_{k+1}} = abab $
  hence Lemma \ref{mp09} gives
  \begin{align*}
 \varphi\big( (ab)^{\xi_1} (ab)^{\xi_2} \cdots & (ab)^{\xi_n}
  \big) =
  \varphi\big(
   (ab)^{\xi_1} \cdots (ab)^{\xi_{k-1} } abab (ab)^{\xi_{k+2}} \cdots (ab)^{\xi_n}
    \big)\\
    =&
 \varphi\big(
    (ab)^{\xi_1} \cdots (ab)^{\xi_{k-1} } ab\big)
    \cdot
    \varphi \big( ab (ab)^{\xi_{k+2}} \cdots (ab)^{\xi_n}
        \big)\\
     =&
  \varphi\big(
     (ab)^{\xi_1} \cdots (ab)^{\xi_{k-1} } ab\big)
     \cdot
     \varphi ( a ) \cdot \varphi \big( b (ab)^{\xi_{k+2}} \cdots (ab)^{\xi_n}
         \big)     =0 .
  \end{align*}
  A similar argument gives that if
  $ \xi_k = \xi_{k+1} = \ast $,
  then
  \begin{align*}
 \varphi\big( (ab)^{\xi_1} (ab)^{\xi_2} \cdots & (ab)^{\xi_n}
   \big) =
   \varphi\big(
        (ab)^{\xi_1} \cdots (ab)^{\xi_{k-1} } b^\ast\big)
        \cdot
        \varphi ( a^\ast ) \cdot \varphi \big( b^\ast a^\ast (ab)^{\xi_{k+2}} \cdots (ab)^{\xi_n}
            \big),
  \end{align*}
hence the conclusion.
\end{proof}

\subsection{Random matrices with non-commutative entries}
We shall use the notation $ M_N ( \mathcal{A}) $ for the the algebra of $ N \times N $ matrices with entries from $ \mathcal{A} $,
   i.e. $ M_N( \mathcal{A} )  = M_N ( \mathbb{C} ) \otimes \mathcal{A} $. The elements of $ M_N ( \mathcal{A} ) $ will be called $ N \times N $ random matrices with entries in $ \mathcal{A} $.
  We refer to \cite{ms, ns, Tao} for further information on random matrix theory.

   The algebra $ M_N( \mathcal{A} ) $  together with the matrix adjoint and the unital positive map $ \varphi \circ \text{tr} $, has a non-commutative probability space structure. (Here $ \text{tr} $ denote the normalized matrix trace.)

\begin{defn}
Suppose that for each positive integer $ N$,
 $\{ X_N(k):\ k \in [ n ] \} $,
  is a set of $ N \times N $ random matrices with entries from $ \mathcal{A} $.

  The family
  $\{ X_N(k):\ k \in [ n ] \}_{N \geq 1} $
 is said to be asymptotically Boolean independent
 if there exists a non-commutative probability space
 $ (\mathcal{B}, \psi) $
 such that $ \mathcal{B} $ is a $ \ast$-algebra
 and there exists a Boolean independent family
 $(b_1, b_2, \dots b_n)$
  from $ \mathcal{B} $ such that
 for every non-commutative
polynomial $p$ in $2n$ variables and complex coefficients, we have that
\[
\lim_{N \rightarrow\infty}
\varphi \circ \text{tr}\big(p(B_N(1), B_{N}(1)^\ast,
\dots, B_N(n), B_N(n)^\ast\big) =   \psi\big(p(b_1, b_1^\ast, \dots, b_n, b_n^\ast) \big).
\]
\end{defn}


\section{Non-self-adjoint random matrices and $B$-diagonal variables}


\begin{defn}\label{iota:1}
 For each $ \oxi = ( \xi_1, \xi_2, \dots, \xi_n) $
 with
 $ \xi_k \in \{ 1, \ast \} $
 and
 $ \oi = (i_1, i_2, \dots, i_n) $
 with $ i_k \in [ N ] $  $( 1 \leq k \leq n ) $,
 we define the interval partition
 $$ \iota( \oxi, \oi) = \big[ l(1), l(2), \dots, l(r) \big] $$
 from $ \mathcal{I}(n) $ as follows.
 Put $ l(0) = 0 $ and $ i_{ n + 1} = i_1 $. Then, inductively,  take
 \begin{align*}
 l(k)= \sup \{ t \in [ n ] : \
  (i_s,  i_{ s+ 1} )^{ \xi_s} = (i_{s+1}, i_{s+2})^{\xi_{s+1}}\ \
  \text{for  all } l(k-1)< s \leq t
 \},
 \end{align*}
 where, for
  $i,j \in [N] $,
   we denote
   $ (i, j)^1 =( i, j) $
   and
   $(i, j)^\ast = (j, i) $.
 \end{defn}

 \begin{defn}
Suppose that
$\sigma $ is an interval partition from
   $ \mathcal{I}(n) $
   and that
$
 \overrightarrow{\xi} = (\xi_1, \xi_2, \dots, \xi_n ) $
 where each
  $ \xi_j $
  is either $ 1 $ or $ \ast $.
 A block
  $( d+1, d+2, \dots, d+ p ) $
   of
 $ \sigma $
 is said to be
 \emph{$ \overrightarrow{\xi}$-alternating}
  if
 $ \xi_{d+1} \neq \xi_{d+2}, \,..., \,\xi_{d+p-1} \neq \xi_{d+p},\;  \xi_{d+p} \neq \xi_{d+1} $.
 If all blocks of $ \sigma $ are
 $ \overrightarrow{\xi}$-alternating,
 then the partition
  $ \sigma $ is said to be
 \emph{$ \overrightarrow{\xi}$-alternating}. The set of all $ \oxi$-alternating interval partitions will be denoted by $ \alt(\oxi)$.
 \end{defn}

\begin{lemma}\label{lemma:4}
Let
$ \sigma = [ l(1), l(2), \dots, l(r) ] \in \mathcal{I}(n) $.
With the notations from above, if $ \sigma $ is $ \oxi $-alternating, we have that
\[
\#\{ \oi\in [ N ]^n : \iota( \oxi, \oi) = \sigma \}
= N^2(N-1)^{r-1}.
\]
If  $ \sigma $ is not $ \oxi $-alternating, then
\[
\#\{ \oi\in [ N ]^n : \iota( \oxi, \oi) = \sigma \}
 \leq N^r.
\]
\end{lemma}

\begin{proof}

Suppose first that $ \sigma $ is $ \oxi $-alternating. Then all the blocks of $ \sigma $ have an even number of elements, so we can put
$p(0) = 0 $
and
$ \sigma = [ 2p(1), 2p(2), \dots, 2p(r)] $
 (in particular, $n = 2p(r) $).
With this notations, we have that
$ \oi \in \iota( \oxi, \overrightarrow{j}) $
if and only if
\begin{equation}\label{eq:xi}
(i_s, i_{s +1})^{\xi_s} = ( i_{s+1}, i_{s+2})^{\xi_{s+1}}
\end{equation}
for  each
 $ k \in \{ 0, 1, \dots, r-1\} $
 and  $ s $ such that
  $ 2p(k)+1 \leq s \leq 2p(k+1) -1 $, and
  \[
  ( i_{ 2p(t)}, i_{ 2p(t) + 1}) \notin \{
  (i_{2p(t) +1}, i_{2p(t) +2}),
  (i_{2p(t) +2}, i_{2p(t) +1})
  \}
  \]
   that is
    $ i_{ 2p(t)}  \neq i_{ 2p(t) +2} $
    for $ t = 1, 2, \dots, r $.

  But each block of $ \sigma $ is $ \oxi $-alternating, so for
  $ 2p(k) + 1 \leq s \leq 2p(k+1) - 1 $
  we have that
  $ \xi_s \neq \xi_{ s + 1 } $,
  hence the equation (\ref{eq:xi}) becomes
  $
   ( i_s, i_{s+1}) = ( i_{s+2}, i_{s+1}) $
   that is
    $ i_s = i_{s+2}$.
    Therefore
     $ \oi \in \iota ( \oxi, \sigma ) $
     is equivalent to the conditions
     \begin{equation}\label{xi:2}
     \left\{
     \begin{array}{l}
     i_{2p(k) + 1} = i_{2p(k) + 3} = \dots = i_{ 2p(k+1)+1}\\
          i_{2p(k)+2}  = i_{2p(k)+4} = \dots = i_{2p(k+1)}\\
       i_{ 2p(t)}  \neq i_{ 2p(t) +2}\\
     \end{array}
     \right.
     \end{equation}
 satisfied for all $ k \in \{ 0, 1, \dots, r-1\} $. Moreover, the first equation from above gives that $ i_q = i_l $ for all $ q, l $ odd elements in $ [ n ] $, therefore, if $ \sigma $ is $ \oxi$-alternating,
 \begin{align*}
 \#\{ \oi\in [ N ]^n & :
  \iota( \oxi, \oi) = \sigma \}   \\
  = &
 \# \{ (i_1, i_{2p(1)}, i_{2p(2)}, \dots, i_{2p(r)}) \in N^{r+1} :
 \text{ $ i_{2p(s)} \neq i_{2p(s+1)} $
  for all $ s \in [ r-1] $} \}\\
   = & N^2 ( N -1)^{ r -1}.
\end{align*}

Next, suppose that
$ \sigma = [ l(1), l(2), \dots, l(r ) ] $
is not $ \oxi$-alternating. Note that,
if
$ D  =(p+1, p+2, \dots, p+q )$
is a block of
$ \iota ( \oxi, \oi) $,
then the tuple
 $(i_{p+1}, i_{p+2}, \dots, i_{p+q+1}) $
 is uniquely determined, according to Definition \ref{iota:1},
  by any of the couples
 $ (i_{p+k}, i_{p+k+1} ) $.
 In particular
 \begin{equation}\label{n2:1}
 \#\{
 (i_{p+1}, i_{p+2}, \dots, i_{p+q+1}) \in [N]^{q+1}: \oi \in \iota( \oxi, \oi)
 \}
 \leq N^2.
 \end{equation}

  Now we claim that for $ t \leq r $, there has
  \begin{equation}\label{eq:iota:1}
  \#\{ (i_1,i_2,\dots, i_{ l(t) + 1}):
   \ \iota(\oxi, \oi) = \sigma\} \leq N^{ t + 1}.
  \end{equation}
  Indeed, we shall proved it by inductive on $t$. 
 For $t = 1$, equation (\ref{eq:iota:1}) follows trivially from (\ref{n2:1}).
   Suppose (\ref{eq:iota:1}) true for $ t \leq s $.
  Then, if the tuple
  $ (i_1,i_2,\dots,i_{ l(s) + 1}) $
  is fixed
  such that
  $ \iota( \oxi, \oi ) = \sigma $,
  we have that in the pair
  $ \{ i_{ l(s) + 1}, i_{l(s) + 2}\} $
   the element
    $ i_{ l(s) + 1 } $
    is fixed, hence there are at most $ N $ such pairs with the property that
     $ \iota(\oxi, \oi) = \sigma $,
     so
     \[
    \#\{ (i_1,i_2,\dots, i_{ l(s+1) + 1}):
       \ \iota(\oxi, \oi) = \sigma\}
       \leq N \cdot
   \#\{ (i_1,i_2,\dots, i_{ l(s) + 1}):
         \ \iota(\oxi, \oi) = \sigma\}
     \]
     hence (\ref{eq:iota:1}) follows by induction.

     Without restricting the generality, we can suppose that the last block of $ \sigma $,
     $ V =  ( l(r-1) +1, l(r-1)+ 2, \dots, l(r) = n) $
     is not $ \oxi $-alternating. Then we have two possible cases.

     Case 1:
     $ \xi_s =\xi_{s+1} $
     for some
     $ s \in V \setminus \{n \} $.
     Then
      $(i_s, i_{s+1})^{\xi_s} = (i_{s+2}, i_{s+1})^{\xi_{s+1}} $
            so $ i_s = i_{s+1} $.
       Hence, as seen in the argument for (\ref{n2:1}),   it follows that
       \[ i_{l(r-1) + 1}= i_{l(r-1)+2} = \dots = i_{l(r)},\]
so, equation (\ref{eq:iota:1}) gives
\[
 \# \{ \oi \in [ N]^n: \iota( \oxi, \oi) = \sigma \}  = \# \{ (i_1, i_2, \dots, i_{l(r-1) + 1}:   \iota( \oxi, \oi) = \sigma \} \leq N^r.
 \]

 Case 2: 
    $ \xi_s  \neq \xi_{s+1} $
      for all
      $ s \in V \setminus \{n \} $
      and $ V $ has an odd number of elements. 
      
     According to equation \ref{eq:xi}, we have $( i_s, i_{s+1}) = ( i_{s+2}, i_{s+1}) $ for each
      $s\in \{l(r-1)+1,l(r-1)+2,\ldots,n-1\}$, i.e. 
      
      \begin{equation*}
     \left\{
     \begin{array}{l}
     i_{l(r-1)+1} = i_{l(r-1)+ 3} = \dots = i_{ n},\\
          i_{l(r-1)+2}  = i_{l(r-1)+4} = \dots = i_{n-1},\\
       i_{ n-1}  \neq i_{ 1}.\\
     \end{array}
     \right.
     \end{equation*}
  The last equation due to $(i_{n+1},i_n)=(i_1,i_n)$. Thus, the tuple 
  $$(i_{l(r-1)+1},i_{l(r-1)+2},\ldots,i_{n-1},i_{n})$$ is fixed by
  tuple $(i_{1},i_{2},\ldots,i_{l(r-1)+1})$.
     Utilizing again equation (\ref{n2:1}), we obtain
     \[
     \#\{ \oi\in [ N ]^n : \iota( \oxi, \oi) = \sigma \} =
     \# \{ (i_1, i_2, \dots, i_{l(r-1) + 1}):
       \iota( \oxi, \oi) = \sigma \}
        \leq N^r.
     \]
 \end{proof}

\begin{thm} \label{thm:4.1}
 For each positive integer $ N $, let $X_N=[x_{ij,N}]_{1\leq i,j\leq N}$
 be a $ N \times N $
  random matrix with  identically distributed and
Boolean independent entries from  $(\mathcal{A},\varphi)$
 such that, for each
  $(i,j)\in [N]\times[N]$:
  \begin{enumerate}
  \item[1.] the limits
  \begin{align*}
  \alpha_m &= \lim_{N \rightarrow \infty}
   N \varphi \big( (x_{ij, N} x_{ij, N}^\ast)^m\big) \\
  \beta_m &= \lim_{N \rightarrow \infty}
   N  \varphi \big( (x_{ij, N}^\ast x_{ij, N})^m\big)
     \end{align*}
  exist for each positive integer $ m $
  \item[2.] $ \displaystyle
   \lim_{N \rightarrow \infty }N^\epsilon
   \varphi\big(
   x_{ i j, N}^{\xi_1} x_{ i j, N}^{\xi_2} \cdots x_{ i j, N}^{\xi_n}
   \big)
   =0,
   $\  \  \
   for all
   $ \varepsilon < 1 $
 all
 $ n \geq 1 $
 and all $n $ -tuples
    $  ( \xi_1, \xi_2, \dots, \xi_n ) \in \{ 1, \ast \}^n $.
  \end{enumerate}
Then the asymptotic distribution of
$ X_N $ is $ B $-diagonal of determining sequences
$( \alpha_n)_n $ and $ ( \beta_n)_n $.
 \end{thm}

\begin{proof}
 In order to simplify the writing, we shall introduce several notations. If $ \sigma $ is an interval partition from $\alt(\oxi) $ and
  $ B = ( p+1, p+2, \dots, p+q ) $
   is a block of $ \sigma $, then we say that $ B \in \sigma^+ $ if
    $ \xi_{p+1} = \ast $
     and
     $ B \in \sigma^- $
     if $ \xi_{p+1} = 1 $.
      With this notation we have that
       $X  \in \cA $
        is  $ B $-diagonal of determining sequences
         $(\alpha_n)_n $
          and $( \beta_n)_n $ if and only if
          \[ \varphi \circ \tr \big( X^{\xi_1} X^{ \xi_2} \cdots X^{ \xi_n}\big)
          = \sum_{ \sigma \in \alt (\oxi) } \prod_{ B \in \sigma^+} \beta_{ \#(B)} \cdot \prod_{ B \in \sigma^-} \alpha_{\#(B)}
          \]
          holds true for all positive integers $ n $ and all
           $ \oxi \in \{ 1, \ast\}^n $.

 Let $ \oxi = ( \xi_1, \xi_2, \dots, \xi_n ) \in \{ 1, \ast\}^n $.
 Then
 \begin{align*}
 \varphi\circ \tr \big( X^{\xi_1}_N X^{\xi_2}_N \cdots X^{\xi_n}_N \big)
  = &
  \frac{1}{N}
  \sum_{ \oi \in [ N]^n}
   \varphi \big(
    x_{i_1 i_2, N}^{(\xi_1)} x_{i_2 i_3, N}^{( \xi_2)} \cdots x_{i_n  i_1, N}^{(\xi_n)} \big)\\
    = &
     \sum_{ \sigma \in \mathcal{I}_n}
      \frac{1}{N}
 \sum_{
 \substack{ \oi \in [ N]^n \\ \iota( \oxi, \oi) = \sigma }
  }
   \varphi_{\sigma}\big[
 x_{i_1 i_2, N}^{(\xi_1)}, x_{i_2 i_3, N}^{( \xi_2)},
 \dots, x_{i_n  i_1, N }^{(\xi_n)}
      \big],
   \end{align*}
where for each $k \in [n]$ and $i,j \in [N]$
\begin{equation*}
  x_{i j, N }^{(\xi_k)}= \left\{
   \begin{array}{l l}
    x_{j  i, N } &  \text{if} \ \xi_{k} = \ast,\\
    x_{i  j, N } &  \text{if} \ \xi_{k} = 1.
   \end{array}
   \right.
 \end{equation*}

 Let
 $ \sigma \in I_n $.
 If
 $ B = ( p+1, p+2, \dots, p+q ) $
 is a block of $ \sigma $, let us denote
 \[ v_{\oxi, N}(B)
  = \varphi \big(
   x_{i j, N}^{(\xi_{p+1})} x_{i j, N}^{(\xi_{p+2})}  \cdots x_{i j, N}^{(\xi_{p+q})}  \big).
   \]
Note that
 $ v_{ \oxi, N}(B) $
  does not  depend on the choice of $ i $ and $j $, since
         $ x_{i j, N } $ are identically distributed.
Henceforth,
    \[
    \sum_{  \iota( \oxi, \oi)  = \sigma }
    \varphi_{\sigma}\big[
     x_{i_1 i_2, N}^{(\xi_1)}, x_{i_2 i_3, N}^{( \xi_2)}, \dots, x_{i_n  i_1, N}^{(\xi_n)}
          \big] =( \# \{ \oi \in [ N]^n: \ \iota( \oxi, \oi) = \sigma \} ) \prod_{ B \in \sigma} v_{ \oxi, N} ( B ).
      \]

          Moreover, condition 2. gives that
         \[
         \lim_{ N \rightarrow \infty } N^{\varepsilon} \cdot  v_{ \oxi, N, }( B) = 0
         \]
 and condition 1. gives that if $ B $ is an $ \oxi$-alternating block,
   \[
   \lim_{ N \rightarrow \infty} N \cdot v_{ \oxi, N, }( B)
   = \left\{
   \begin{array}{ll}
   \alpha_{\# B } , & \textrm{ if } B \in \sigma^-\\
   \beta_{ \# B } , & \textrm{ if } B \in \sigma^+.
   \end{array}
   \right.
   \]

Therefore, if $ \sigma $ is not $\oxi$-alternating,  Lemma \ref{lemma:4} gives that
\begin{align}
 \frac{1}{N}\sum_{  \iota( \oxi, \oi)  = \sigma }
   \varphi_{\sigma}\big[
    x_{i_1 i_2, N}^{(\xi_1)}, x_{i_2 i_3, N}^{( \xi_2)}, \dots, x_{i_n  i_1, N}^{(\xi_n)}
         \big]
         \leq N^{\# \sigma -1 }\prod_{B \in \sigma }v_{ \oxi, N} ( B ) \label{eq:blocks1}\\
         =
       \prod_{B \in \sigma } N^{ \frac{\# \sigma -1}{\# \sigma} }  \cdot v_{ \oxi, N} ( B )
       \xrightarrow[N \rightarrow\infty]{} 0.  \nonumber
\end{align}

On the other hand, if $ \sigma $ is $ \oxi$-alternating, then
Lemma \ref{lemma:4} gives that
  \begin{align}
\lim_{N \rightarrow \infty}   \frac{1}{N}\sum_{  \iota( \oxi, \oi)  = \sigma }
     \varphi_{\sigma}\big[
      x_{i_1 i_2, N}^{(\xi_1)}, x_{i_2 i_3, N}^{( \xi_2)}, \dots, x_{i_n  i_1, N}^{(\xi_n)}
           \big]
           = \lim_{N \rightarrow \infty}
            N^{\# \sigma}\prod_{B \in \sigma }v_{ \oxi, N} ( B )\label{eq:blocks}\\
           =
           \lim_{N \rightarrow\infty}
         \prod_{B \in \sigma } N  \cdot v_{ \oxi, N} ( B )
          = \prod_{ B \in \sigma^-} \alpha_{\#B}
          \cdot \prod_{B \in \sigma^+} \beta_{\#B}, \nonumber
  \end{align}
  hence the conclusion.
\end{proof}

 To simplify the writing of the proof of the next result, we introduce several notations.
  Suppose that
  $ \sigma \in \mathcal{I}(n) $
   and
 $ \oxi = (\xi_1, \dots, \xi_n) \in \{ 1, \ast\}^n $.
 If
  $D = \{ l+1, l+2, \dots, l+p \} $
  is a subset of
  $ [ n ] $,
  we denote by
  $ \oxi_{ | D }
   = ( \xi_{ l+1}, \xi_{l+2}, \dots, \xi_{ l+p}) $
   and by
   $ \sigma_{| D }$
   the interval partition in
   $ \mathcal{I}(p) $
  given as follows: $ B $ is a block of
$ \sigma_{| D }$
if and only if there exists $ B_1 $ a block of
$ \sigma $
such that $ B = B_1 \cap D $.

\begin{thm} \label{thm:4.2}

Let $ m $ be a positive integer. Suppose that for every positive integer $ N $ and for every
$ k \in [ m ] $,
$ X(k)_N = [ x(k)_{ ij, N} ]_{1 \leq i, j\leq N} $ is a
 $ N \times N $
  matrix with entries from
$\mathcal{A} $
such that
$ \{ x(k)_{i, j, N} :\ i, j \in [ N ] ,
 k \in [ m ]\} $
 form a identically distributed and Boolean independent family (i.e. different entries from the same matrix as well as entries from different matrices are identically distributed and Boolean independent)
 and each
 $ x(k)_{i j, N} $
 satisfies the conditions \emph{1.} and \emph{2.}
 from Theorem \ref{thm:4.1}.

 Under the conditions above, the family
 $ \{ X(k)_N : k \in [ m ] \} $
  is asymptotically Boolean independent.

\end{thm}

\begin{proof}
  If
  $ \oik =
  ( k_1, k_2, \dots, k_n ) $
   is an
   $ n $-tuple with components in $ [m ]$,
   we define the interval partition
   $ \omega(\oik)  =
    [ t(1), t(2), \dots, t(s)]
    \in \mathcal{I}(n) $
 as follows.
  Let $ t(0) = 0 $ and, inductively,
  \[
  t(s) = \max\{ v:\ k_i = k_j \text{ for all } t(s-1) + 1 \leq i, j \leq v \},
   \]
 i.e.
  $ \omega(\oik) $
  is the maximal element of
   $ \mathcal{I}(n) $
   such that the components of
    $ \oik $
    are constant on its blocks. With this notation, it suffices to show that
   for any
    $ \oik = (k_1, k_2, \dots, k_m) \in [ m ]^n $
    and any
    $ \oxi = ( \xi_1, \dots, \xi_n ) \in \{ 1, \ast\}^n $
    we have that
    \begin{align}
    \lim_{ N \rightarrow\infty} &
     \varphi \circ \tr
     \big(
     X(k_1)_N^{\xi_1} \cdot  X(k_2)_N^{\xi_2}
     \cdots
     X(k_n)_N^{\xi_n}
     \big)
      \label{eq:13}\\
      =  &
     \prod_{ \substack{ B \in \omega( \oik)\\ B = (l+1, \dots, l+p)}}
     \lim_{ N \rightarrow\infty}
          \varphi \circ \tr
          \big(
          X(k_{l+1})_N^{\xi_{l+1}} \cdot  X(k_{l+2})_N^{\xi_{l+2}}
          \cdots
          X(k_{l+p})_N^{\xi_{l+p}}
          \big).
          \nonumber
    \end{align}

On the other hand, the definition of $ \tr $ gives
 \begin{align*}
 \varphi \circ \tr
 \big(
 X(k_1)_N^{\xi_1} \cdot &  X(k_2)_N^{\xi_2}
 \cdots
 X(k_n)_N^{\xi_n}
 \big)\\
  & =  \frac{1}{N}
  \sum_{ \oi \in [N]^n }
  \varphi\big(
  x(k_1)_{i_1i_2, N}^{(\xi_1)} \cdot
   x(k_2)_{i_2i_3, N}^{(\xi_2)}
  \cdots
  x(k_n)_{i_ni_1, N}^{(\xi_n)}
  \big).
 \end{align*}
Since, for
 $ k_p \neq k_q $,
  all
   $ x(k_p)_{ij, N} $
    are   Boolean independent from all
    $ x(k_q)_{ij, N} $,
    we have that
    \begin{align*}
   \varphi\big(
   x(k_1)_{i_1i_2, N}^{(\xi_1)} \cdot
    x(k_2)_{i_2i_3, N}^{(\xi_2)} &
   \cdots
   x(k_n)_{i_ni_1, N}^{(\xi_n)}
   \big) \\
   & =
   \varphi_{ \omega( \oik)}
   \big[
 x_{i_1i_2, N}(k_1)^{(\xi_1)} ,
    \dots,
    x_{i_ni_1, N}(k_n)^{(\xi_n)}
   \big].
    \end{align*}
 Therefore, using the fact that
  $ x(k)_{i_1j_1, N} $
  and
  $ x(k)_{i_2 j_2, N} $
  are Boolean independent whenever
 $ (i_1, j_1) \neq (i_2 , j_2) $,
 with the notations from the proof of Theorem
 \ref{thm:4.1},
 we obtain that
\begin{align*}
\varphi \circ \tr
 \big( &
 X(k_1)_N^{\xi_1} \cdot   X(k_2)_N^{\xi_2}
 \cdots
 X(k_n)_N^{\xi_n}
 \big)\\
 = & \frac{1}{N}
 \sum_{ \sigma \in \mathcal{I}(n)}
  \sum_{
  \substack{ \oi \in [ N]^n \\ \iota( \oxi, \oi) = \sigma }
   }
    \varphi_{\sigma \wedge \omega( \oik)}
    \big[
  x(k_1)_{i_1i_2, N}^{(\xi_1)} ,
     \dots,
     x(k_n)_{i_ni_1, N}^{(\xi_n)}
    \big] \\
    =  & \sum_{ \sigma \in \mathcal{I}(n)}
  \frac{1}{N}
 ( \# \{ \oi \in [ N]^n: \ \iota( \oxi, \oi) = \sigma \} ) \prod_{ B \in \sigma\wedge \omega(\oik)} v_{ \oxi, N} ( B ).
\end{align*}

Remark that, from the definitions of
$ \omega( \oik) $
and
$ \iota(\oxi, \oi) $,
if $ B $ is a block of
$ \sigma \wedge \omega(\oik) $
and
$ \iota( \oxi, \oi) = \sigma $
then
$ \{ x(k_l)_{ i_l i_{l+1}}^{ (\xi_l)}:\ l \in B \} $
satisfy the conditions from Theorem \ref{thm:4.1}.
Henceforth, if
 $ \sigma \wedge \omega(\oik) $
is not $ \oxi $-alternating, equation (\ref{eq:blocks1}) gives
\[
\lim_{ N \rightarrow \infty}
 \frac{1}{N}
( \#
\{ \oi \in [ N]^n: \ \iota( \oxi, \oi) = \sigma \} )
\cdot
 \prod_{ B \in \sigma\wedge \omega(\oik)} v_{ \oxi, N} ( B )
=0,
\]


On the other hand, if
$ \sigma \wedge\omega(\oik) $
is $\oxi$-alternating, equation (\ref{eq:blocks}) gives
\begin{align*}
\lim_{N \rightarrow \infty}
\frac{1}{N}
\sum_{
  \substack{ \oi \in [ N]^n \\ \iota( \oxi, \oi) = \sigma }
   }
    \varphi_{\sigma \wedge \omega( \oik)}
    \big[
  x(k_1)_{i_1i_2, N}^{(\xi_1)} ,
     \dots,
     x(k_n)_{i_ni_1, N}^{(\xi_n)}
    \big]
    =
    \prod_{ B \in \sigma \wedge \omega(\oik)}
 w( \oxi, B)
\end{align*}
where
\[
 w( \oxi, B)=
    \lim_{N \rightarrow\infty}
     N  \cdot v_{ \oxi, N} ( B ).
\]
Therefore the left-hand side of (\ref{eq:13}) equals
\[
\sum_{
 \substack{ \sigma \in \mathcal{I}(n)\\
 \sigma \wedge \omega(\oik) \in \alt(\oxi) } }
 \prod_{ B \in \sigma \wedge \omega(\oik)}
 w ( \oxi, B )
\]
which, since each block of
$ \omega(\oik) $
is a union of blocks of
$ \sigma \wedge \omega(\oik) $,
equals
\begin{equation}\label{15}
\sum_{
\substack{ \sigma \in \mathcal{I}(n)\\
 \sigma \wedge \omega(\oik) \in \alt(\oxi) } }
 \big[
\prod_{ D \in \omega(\oik)}
\big(
\prod_{ B \in \sigma_{| D } }
 w ( \oxi, B)
 \big)
 \big].
\end{equation}

Similarly, if
$ D = (l+1, l+2, \dots, l+ p ) $
is a block of
$ \omega( \oik ) $,
we have that
\begin{align*}
\lim_{N \rightarrow \infty}
\varphi \circ \tr
\big(
X(k_{l+1})_N^{\xi_{l+1}} \cdots X(k_{l+p})_N^{\xi_{l+p}}
\big)
=
\sum_{ \sigma \in \mathcal{I}_p }
\big(
\prod_{ B \in \sigma }
w ( \oxi_{| D }, B )
\big)
\end{align*}
hence the right-hand side of (\ref{eq:13}) equals
\begin{equation}\label{16}
\prod_{ D \in \omega( \oik) }
\big[
\sum_{ \sigma \in \alt ( \oxi_{ | D }) }
\big(
\prod_{ B \in \sigma }
w( \oxi_{ | D }, B )
\big)
\big].
\end{equation}
 But
 $ \displaystyle \tau  = \underset{ D \in \omega ( \oik)  }{\oplus }  \tau_{ | D } $
 for all
 $\tau \in \mathcal{I}_n $,
 hence
\[
\{ \sigma \in \mathcal{I}_n : \sigma \wedge \omega(\oik) \in \alt ( \oxi) \} =
\{ \sigma \in \mathcal{I}_n :
\sigma_{ | D } \in \alt ( \oxi_{ | D })
\textrm{ for all } D \in \omega( \oik )
\}
\]
 so the expression (\ref{15}) and (\ref{16}) are equal and the conclusion follows.
\end{proof}


\section{Permutations of entries and asymptotic Boolean independence}

 Denote by
  $ \mathcal{S}([ N ]^2) $
the set of all bijections on
$ [ N ] \times [ N ] $
and by $ e $ the identity element in
$ \mathcal{S}([ N ]^2) $
(i.e. $ e(i, j) = ( i, j) $ for all $ i, j \in [ N ] $).
We define the involution
$ \alpha \mapsto \alpha^\ast $
on
$ \mathcal{S}([ N ]^2) $
via
$ \alpha^\ast ( i, j) = \alpha ( j, i ) $
for all $ i, j \in [ N ] $.

\begin{defn}\label{defn:5.1}
  Let $ n $ be a positive integer. Suppose that
 $ \oia = ( \alpha_1, \alpha_2, \dots, \alpha_n) $,
 that
 $ \oxi = ( \xi_1, \xi_2, \dots, \xi_n) $,
 and that
 $ \oi = (i_1, i_2, \dots, i_n) $
 where, for each $ k $, we have that
  $ \alpha_k \in \mathcal{S}( [ N ]^2 ) $,
 $ \xi_k \in \{ 1, \ast \} $
 and $ i_k \in [ N ]$.

\noindent \emph{(i)} We will denote by
 $ \alt ( \oia, \oxi ) $
  the set of
   $ \sigma \in \mathcal{I}(n) $,
such that
 $ \sigma $ is
  $ \oxi $-alternating
   and
   $ \alpha_k = \alpha_l $
   whenever $ k $ and $ l $ are in the same block of $ \sigma $.

 \noindent\emph{(ii)} We will denote by
 $ \iota ( \oia, \oxi, \oi ) $
 the interval partition
 $     [ l(1), l(2), \dots, l(r)] $
from
 $ \mathcal{I}(n) $
 defined  as follows.
 Put
 $ l(0) = 0 $
 and
 $ i_{ n +1} = i_1 $.
 Then, inductively, take
 $
 l(k) = \sup\{ t \in [ n ] :\
  \alpha_s^{ \xi_s } ( i_s, i_{s+1})
   =
  \alpha_{s+1}^{ \xi_{s+1} }
  ( i_{s+1}, i_{s+2})
  \text{ for  all } l(k-1)< s \leq t
    \}.
 $
\end{defn}

\begin{lemma}
Suppose that for each $ N $,
$ \alpha_N $
 is a permutation from
 $\mathcal{S} ( [ N ]^2 ) $,
 and that for all $ \theta < 2 $
 \begin{equation}\label{theta}
 \lim_{ N \rightarrow \infty }
 \frac{1}{ N^{ \theta } }
  \cdot
 \# \big\{
  (i, j , k ) \in [ N ]^3 : \
 \alpha_N (i, j) \in \{ (j, k ), (k, i)\}
  \big\} =0.
 \end{equation}
 \noindent
Suppose
that
$ \sigma  $
 is an element of
 $ \mathcal{I}(n) $
and that
$ \oxi = ( \xi_1, \xi_2, \dots, \xi_n ) $
with
$ \xi_j \in \{ 1, \ast \} $.
Moreover, suppose that for each positive integer $ N $,
$ \oia (N) = ( \alpha_1(N), \alpha_2(N), \dots, \alpha_n(N) ) $
is such that, for each $ j $, either
$ \alpha_j(N)  = e $ for all $ N $,
or $ \alpha_j(N ) = \alpha_N $
for all $ N $.

With these notations, we have that
\begin{enumerate}
\item[(i)] If
$ \sigma \in \alt ( \oia, \oxi) $,
then
$
\# \{\oi \in [ N ]^n:\
\iota( \oia, \oxi, \oi) = \sigma
\} = N^{ r +1 }.
$
\item[(ii)] If
$ \sigma \notin \alt ( \oia, \oxi) $,
then, for all
 $  \theta \in (1, 2) $
\end{enumerate}
\[
\lim_{ N \rightarrow \infty}
\frac{1}{ N^{ \theta + r - 1 }} \cdot
\# \{\oi \in [ N ]^n:\
\iota( \oia, \oxi, \oi) = \sigma
\} = 0
\]
\end{lemma}

\begin{proof}

Suppose first that
$ \sigma \in \alt ( \oia, \oxi) $.
In particular, all blocks of $ \sigma $ have an even number of elements. Also, the condition
$ \iota( \oia, \oxi, \oi) = \sigma $
is equivalent to
\[
\alpha_s^{ \xi_s } ( i_s, i_{s+1})
   =
 \alpha_{s+1}^{ \xi_{s+1} }
  ( i_{s+1}, i_{s+2})
\]
whenever $ s $ and $ s+1 $ are in the same block of $ \sigma $. Then
 $ \alpha_s = \alpha_{ s + 1 } $
 and
 $ \xi_s \neq \xi_{ s + 1 } $
 so the equation above reads
 \[
  \alpha_s ( i_s, i_{s + 1 }) =
  \alpha_{s+1} ( i_{s+2}, i_{s+1}),
 \]
  which, since all
 $ \alpha_s $
  are bijections,
 means that
 $ i_{s } = i_{s+2} $.
 Hence, part (i) follows from the argument in the proof of Lemma \ref{lemma:4}.

 If
 $ \sigma \notin \alt ( \oia, \oxi, \oi ) $
 but the components of $ \oia $ are constant on the blocks of $ \sigma $, then
 $ \sigma \not \in \alt ( \oxi ) $
 and the argument from the proof of Lemma \ref{lemma:4} gives that
 \[
 \# \{ \oi \in [ N ]^n: \ \iota( \oia, \oxi, \oi) = \sigma \} \leq N^r.
 \]

Suppose that
$ \sigma \notin \alt( \oia, \oxi) $
and there is a block of $ \sigma $ on which the components of $ \oia $ are not constant.
Since we can permute circularly the blocks of $\sigma $, we can suppose that the block with this property is the one containing 1.
Denote
$ \sigma = [ l(1), l(2), \dots, l(r)] $
and suppose that
$ \alpha_s \neq \alpha_{s+1} $
for some
$ 1 \leq s \leq l(2)-1 $. Then
\[
\alpha_s^{ \xi_s } ( i_s, i_{s+1})
   =
 \alpha_{s+1}^{ \xi_{s+1} }
  ( i_{s+1}, i_{s+2})
\]
 gives that
 \[
 (i_s, i_{s+1})^{ \xi_s }
 = \alpha_s^{ -1} \circ \alpha^{\xi_{s+1}}_{s+1}
 \big( (i_{s+1}, i_{s+2})  \big)
 \]
so condition (\ref{theta}) gives that,
for all
$ \theta < 2 $
\[
\lim_{ N \rightarrow \infty}
 \frac{1}{N^\theta} \cdot
 \#\{ (i_s, i_{s+1}):\
  \iota( \oia, \oxi, \oi) = \sigma\}   = 0.
\]
  But Definition \ref{defn:5.1}(ii) gives that the pair
   $ (i_{ s}, i_{s+1} )$
  uniquely determines the tuple
   $ ( i_{1}, i_{  2 },  \dots, i_{l(1) + 1}) $,\
   hence
\begin{equation}\label{theta:2}
\lim_{ N \rightarrow \infty}
 \frac{1}{N^\theta} \cdot
 \#\{ (i_1, i_2, \dots, , i_{ l(1) + 1 }):\
  \iota( \oia, \oxi, \oi) = \sigma\}   = 0.
\end{equation}
On the other hand,   if the tuple
  $ (i_1,i_2,\dots,i_{ l(s) + 1}) $
  is fixed  such that
  $ \iota( \oia,  \oxi, \oi ) = \sigma $,
  we have that in the pair
  $ (i_{ l(s) + 1}, i_{l(s) + 2} )$
   the element
    $ i_{ l(s) + 1 } $
    is fixed, hence there are at most $ N $ values of
     $ i_{l(s) + 2 } $
 such that
     $ \iota(\oia, \oxi, \oi) = \sigma $.
Again, Definition \ref{defn:5.1}(ii) gives that the pair  $ (i_{ l(s) + 1}, i_{l(s) + 2} )$
 uniquely determines the tuple
  $ ( i_{l(s) + 1}, i_{ l(s) + 2 },
  \dots, i_{l(s+1) + 1}) $., henceforth
\begin{align}
    \#\{ (i_1,i_2,\dots, i_{ l(s+1) + 1}):
       \  & \iota( \oia, \oxi, \oi)  = \sigma\} \label{ineq:5}\\
       &  \leq N \cdot
   \#\{ (i_1,i_2,\dots, i_{ l(s) + 1}):
         \ \iota( \oia, \oxi, \oi ) = \sigma\}.
         \nonumber
\end{align}
Finally, equations (\ref{theta:2}) and (\ref{ineq:5}) give that, for all
$ \theta < 2 $,
\[
\lim_{ N \rightarrow \infty}
\frac{1}{ N^{ \theta + r - 1 }} \cdot
\# \{\oi \in [ N ]^n:\
\iota( \oia, \oxi, \oi) = \sigma
\} = 0
\]
hence the conclusion.

\end{proof}

 Let
 $ \alpha: [ N ] \times [ N ] \rightarrow [ N ] \times [ N ] $
be a bijection. For
 $ A  = [ a_{ i j}]_{ i,j =1}^N$ a
$ N \times N $
matrix with entries from the
$ \ast $-algebra $ \mathcal{A} $,
we will denote
$ A^{\lceil \alpha \rceil } =
[ a_{ \alpha(i, j )}]_{ i, j =1}^N $,
i.e. the $ (i, j) $-entry of
$A^{\lceil \alpha \rceil } $
equals the $ \alpha(i,j) $-entry
of $ A $.
 With this notation, we have the following result.

\begin{thm}\label{thm:5.1}
 For each positive integer $ N $, let $X_N=[x_{ij,N}]_{1\leq i,j\leq N}$
 be a $ N \times N $
  random matrix with identically distributed and
Boolean independent entries from  $(\mathcal{A},\varphi)$
 such that, for each
  $(i,j)\in [N]\times[N]$:
  \begin{enumerate}
  \item[1.] the limits
  \begin{align*}
  \alpha_m &= \lim_{N \rightarrow \infty}
   N \varphi \big( (x_{ij, N} x_{ij, N}^\ast)^m\big) \\
  \beta_m &= \lim_{N \rightarrow \infty}
   N  \varphi \big( (x_{ij, N}^\ast x_{ij, N})^m\big)
     \end{align*}
  exist for each positive integer $ m $
  \item[2.] $ \displaystyle
   \lim_{N \rightarrow \infty }N^\epsilon
   \varphi\big(
   x_{ i j. N}^{\xi_1} x_{ i j. N}^{\xi_2} \cdots x_{ i j. N}^{\xi_n}
   \big)
   =0,
   $\  \  \
   for all
   $ \varepsilon < 1 $
 all
 $ n \geq 1 $
 and all $n $ -tuples
    $  ( \xi_1, \xi_2, \dots, \xi_n ) \in \{ 1, \ast \}^n $.
  \end{enumerate}
  Also, suppose that for each
  positive integer $ N $,
  $ \alpha( N ) $
   is a permutations from $ {S}( [ N ]^2) $
    such that, for all
   $\theta > 2 $,
   \[
  \lim_{ N \rightarrow \infty }
   \frac{1}{ N^{ \theta } }
    \cdot
   \# \big\{
    (i, j , k ) \in [ N ]^3 : \
   \alpha_N (i, j) \in \{ (j, k ), (k, i)\}
    \big\} =0.
   \]
   Then the matrices $ X_N $ and $ X^{\lceil \alpha_N \rceil }_N $ are asymptotically
   (as $ N \rightarrow \infty $)
    Boolean independent $ B $-diagonals of determining sequences
    $( \alpha_n)_n $ and $ ( \beta_n)_n $.
\end{thm}

\begin{proof}
The fact that the asymptotic distributions of the matrices $ X_N $ and
$ X_N^{ \lceil \alpha ( N ) \rceil } $
are both $ B $-diagonal of determining sequences
$( \alpha_n)_n $ and $ ( \beta_n)_n $
is proved in Theorem \ref{thm:4.1}. It only remains to show the asymptotic Boolean independence.

 To simplify the notations, we will omit the index
 $ N $, i.e. we shall write
  $ X $,
 respectively
 $ X^{ \lceil \alpha \rceil } $
 for $ X_N $,
 respectively
 $ X^{ \lceil \alpha ( N ) \rceil }_N $
   with the convention that only matrices of the same size are multiplied. Also, we shall use the notations
   $ X^{\alpha, 1} $,
   respectively
   $X^{ \alpha,  \ast }$
    for
    $ X^{ \lceil \alpha \rceil} $,
    respectively
$ \big( X^{ \lceil \alpha \rceil}   \big)^\ast $.

 For a positive integer $ n $, let
 $ \oia = ( \alpha_1, \alpha_2, \dots, \alpha_n) $
 and
 $ \oxi = ( \xi_1, \xi_2, \dots, \xi_n) $
 where, for each $ j $, we have that
  $ \alpha_j \in \{ 1, \alpha \} $
  and
  $ \xi_j \in \{ 1, \ast \} $.
 It suffices to show that
 \begin{equation}\label{mix:alpha}
 \lim_{ N \rightarrow \infty }
 \varphi \circ \tr
 \big(
 X^{\alpha_1, \xi_1} \cdot
   X^{ \alpha_2, \xi_2 }
   \cdots
   X^{ \alpha_n, \xi_n }
 \big)
 =
 \sum_{ \sigma \in \alt ( \oia, \oxi ) } \
 \big[
 \prod_{ B \in \sigma^-} \alpha_{\#B}
   \cdot
   \prod_{B \in \sigma^+} \beta_{\#B}
   \big].
 \end{equation}

 Using the Boolean independence of the entries of $ X $, we obtain
 \begin{align*}
 \varphi \circ \tr
  \big(
  X^{\alpha_1, \xi_1} \cdot
    X^{ \alpha_2, \xi_2 }
    \cdots
    X^{ \alpha_n, \xi_n }
  \big)
  = \sum_{ \oi \in [ N ]^ n }
  \frac{1}{N}
  \varphi
  \big(
  x_{ \alpha_1( i_1, i_2)}^{( \xi_1 )}
  x_{ \alpha_2( i_2, i_3)}^{( \xi_2 )}
  \cdots
 x_{ \alpha_n( i_n, i_1)}^{( \xi_n )}
  \big)\\
  =\sum_{ \sigma \in \mathcal{I}(n)}
  \frac{1}{N}
  \sum_{
    \substack{ \oi \in [ N]^n \\
     \iota( \oia, \oxi, \oi) = \sigma } }
  \varphi_{\sigma}
   \big[
   x_{ \alpha_1( i_1, i_2)}^{( \xi_1 )},
   x_{ \alpha_2( i_2, i_3)}^{( \xi_2 )},
   \cdots,
  x_{ \alpha_n( i_n, i_1)}^{( \xi_n )}
   \big].
 \end{align*}

 Let
 $ B = (l+1, l+2, \dots, l+p ) $
 be a block of $ \sigma $.
 Then, if
 $ \iota ( \oia, \oxi, \oi ) = \sigma $,
  we have that
  $
   x_{
   \alpha_{l + s}( i_{ l + s}, i_{ l + s + 1 })
   }^{( \xi_{l + s })}
   \in
   \big\{
   x_{ \alpha_{l + 1}( i_{ l + 1}, i_{ l + 2 })},
  x_{ \alpha_{l + 1}( i_{ l + 1},i_{l+2})}^\ast
   \big\}
   $,
  therefore
 \[
 \varphi \big(
 x_{ \alpha_{l + 1}( i_{ l + 1}, i_{ l + 2 })}^{( \xi_{ l + 1})}
 x_{ \alpha_{l + 2}( i_{ l + 2}, i_{ l + 3 })}^{( \xi_{ l + 2})}
 \cdots
 x_{ \alpha_{l + p}( i_{ l + p}, i_{ l + p + 1 })}^{( \xi_{ l + p})}
 \big)
  =
  v_{ \oxi, N }(B).
 \]
 hence
 \begin{align*}
   \frac{1}{N}
   \sum_{
     \substack{ \oi \in [ N]^n \\
      \iota( \oia, \oxi, \oi) = \sigma } }
   \varphi_{\sigma}
    \big[
    x_{ \alpha_1( i_1, i_2)}^{( \xi_1 )}, &
    x_{ \alpha_2( i_2, i_3)}^{( \xi_2 )},
    \cdots,
   x_{ \alpha_n( i_n, i_1)}^{( \xi_n )}
    \big]\\
    & =
   \frac{1}{N} \cdot
    \#\{
    \oi\in [ N ]^n:\ \iota(\oia, \oxi, \oi) = \sigma \}
    \cdot
     \prod_{B \in \sigma} v_{ \oxi, N}(B).
 \end{align*}
If
$ \sigma \notin \alt ( \oia, \oxi, \oi ) $,
fix
$ \theta \in (1, 2) $
and denote
$ \sigma^\prime = \{ B \in \sigma: B = \oxi \textrm{-alternating}\} $,
$ \sigma^{\prime\prime} = \{ B \in \sigma:
 B \notin \oxi \textrm{-alternating}\} $.

We have that
\begin{align*}
\frac{1}{N} \cdot &
    \#\{
    \oi\in [ N ]^n:\
     \iota(\oia, \oxi, \oi) = \sigma \}
    \cdot
     \prod_{B \in \sigma} v_{ \oxi, N}(B)\\
     = &
     \big(
\frac{1}{ N^{ \theta + r - 1 }} \cdot
\# \{\oi \in [ N ]^n:\
\iota( \oia, \oxi, \oi) = \sigma
\}
     \big)
     \cdot
     \big(
   \prod_{ B \in \sigma^\prime} N \cdot v_{ \oxi, N }(B)
     \big)
     \cdot
     \big(
     \prod_{ B \in \sigma^{ \prime \prime}}
     N^{\varepsilon(\sigma)}
  \cdot v_{ \oxi, N }(B)   \big),
\end{align*}
where
$ \displaystyle
\varepsilon(\sigma)=
 \frac{\theta + r -( \# \sigma^\prime +2)}{n}  < 1 $.

  Since, according to Lemma \ref{lemma:4}, respectively condition 2., the fist factor, respectively the third factors in the product above cancel as
  $ N \rightarrow \infty $,
  it follows that, if
  $ \sigma \notin \alt (\oia, \oxi, \oi) $,
 \[
 \lim_{ N \rightarrow \infty}
 \frac{1}{N} \cdot
     \#\{
     \oi\in [ N ]^n:\
      \iota(\oia, \oxi, \oi) = \sigma \}
     \cdot
      \prod_{B \in \sigma} v_{ \oxi, N} = 0.
 \]

 If
 $ \sigma \in \alt( \oia,\oxi, \oi) $,
 then all blocks of $\sigma $ are
 $\oxi$-alternating and Lemma \ref{lemma:4} gives that
 \begin{align*}
 \lim_{ N \rightarrow \infty}
  \frac{1}{N} \cdot
      \#\{
      \oi\in [ N ]^n:\
       \iota(\oia, \oxi, \oi) = \sigma \} &
      \cdot
       \prod_{B \in \sigma} v_{ \oxi, N}(B)
       =
        \lim_{ N \rightarrow \infty}
        \prod_{ B \in \sigma}
        N \cdot v_{ \oxi, N}(B)\\
        = &
  \prod_{ B \in \sigma^-} \alpha_{\#B}
     \cdot
     \prod_{B \in \sigma^+} \beta_{\#B}
 \end{align*}
 hence the conclusion.
\end{proof}

  To formulate the following consequence of the Theorem above, we need first to define (following \cite{mp3}) the notion of  partial $m$-transpose of a square matrix.
  \begin{defn}
  Let $m, n $ be two positive integers and
  $ A \in M_{mn}(\mathcal{A})$. We define the \emph{partial $m $-transpose} of $ A $ as follows.
  We see $ A $ as a $ m \times m $ matrix
  $ A = [ A_{ij}]_{i,j=1}^n $,
  with entries in
  $ M_n(\mathcal{A} )$.
  The partial $ m$-transpose of $ A $ is the matrix
  $A^{\Gamma,m} = [ A_{i,j}^T]_{i, j=1}^n $
  obtained by transposing (as a $n \times n $ matrix)
  of the $m^2$ blocks of $ A $.
  \end{defn}

\begin{conseq}
Suppose that
 $\big( m(N)\big)_N $
 and
 $ \big(n(N)\big)_N $
 are two non-decreasing sequences of positive integers.
 Furthermore, suppose that
 $ \big( X_N \big)_N $
 is a sequence of
 $ m(N)n(N) \times m(N)n(N) $ random matrices with   identically distributed and
 Boolean independent entries from  $(\mathcal{A},\varphi)$  that satisfy the conditions 1. and 2. from Theorems \ref{thm:4.1} and
 \ref{thm:5.1}. Then $ X_N $ and its partial $m $-transpose $X_N^{\Gamma, m(N)} $  are asymptotically (as
  $ N \rightarrow \infty $)
   Boolean independent if and only if
   $ \displaystyle \lim_{N \rightarrow \infty} n(N) = \infty $.
\end{conseq}
\begin{proof}
Let
\[ \Delta_{m(N)}= \{ (i, j) \in [ n(N) m(N)]^2: i\equiv j (\text{mod } n(N))\}
\]
i.e.  $ \Delta_{m(N)} $ is the set of all
 entries of the diagonal of some
 $n(N) \times n(N) $
 block of $ X_N $.
 Note that
 \[
 \#\Delta_{m(N)} = n(N) \cdot m(N)^2 = \frac{1}{n(N)}\cdot \big( m(N) n(N)\big)^2.
\]
 Also, if for   each $ N $, let $ \alpha(N) $ be the permutation
  on
  $ [ m(N) n(N)]^2 $ such that
  $ X_N^{\lceil \alpha(N)\rceil} = X_N^{\Gamma, m(N)}$, we have that
  \begin{align*}
  \alpha(N)(i, j)& =  (i, j) \textrm{ if $(i, j)\in \Delta_{m(N)}$}\\
 \alpha(N)(i, j)& \notin \{ (i, k), (k, j) : k \in [ m(N) n(N)] \} \text{  if $(i, j) \notin \Delta_{m(N)} $},
  \end{align*}
therefore
  \[
  \{ (i, j) \in [ m(N)n(N)]^2:\ \alpha(N)(i, j) \in \{ (i, k), (k, j)\}
   \textrm{ for some $ k \in [ m(N)n(N)]$ } \}
    = \Delta_{m(N)}.
  \]

 Thus, if
 $ \displaystyle \lim_{N \rightarrow \infty} n(N) = \infty $, the conclusion follows from  Theorem \ref{thm:5.1}.

 Suppose now that  $ X_N $ and $X_N^{\Gamma, m(N)} $
 are asymptotically Boolean independent. Then
 \[ \lim_{N \rightarrow \infty}
  \varphi \circ \tr \big(X_N^\ast \cdot X_N^{\Gamma, m(N)} \big)
 =
 \lim_{N \rightarrow \infty}
  \varphi \circ \tr\big( X_N^\ast)
 \cdot
 \lim_{N \rightarrow \infty} \varphi \circ \tr
 \big(
 X_N^\ast \cdot X_N^{\Gamma, m(N)}
 \big) = 0.
 \]
 On the other hand,
 \begin{align*}
 \varphi \circ \tr
 \big(
 X_N^\ast \cdot X_N^{\Gamma, m(N)}
  \big)
& = \frac{1}{m(N)n(N)}
 \sum_{i, j\in [m(N)n(N)]}
  \varphi
 ( x_{ij, N}^\ast x_{\alpha(N)(i, j), N})\\
 = &
 \frac{1}{[m(N)n(N)]^2}
 \sum_{(i, j)\in \Delta_{m(N)}}
 [ m(N)n(N)
 \varphi ( x_{ij, N}^\ast x_{ij, N} ) ]\\
  = &
  \frac{1}{n(N)}[ m(N)n(N)
  \varphi ( x_{ij, N}^\ast x_{ij, N} ) ]
 \end{align*}
  and the conclusion follows from condition 1. of Theorem \ref{thm:5.1}.
\end{proof}


\section{Self-adjoint random matrices and Bernoulli distributed non-commutative variables}


\begin{thm}
 Let
  $ B_N  = [ b_{ij, N}]_{i, j=1}^N $
  be a self-adjoint matrix in
   $ M_N(\mathcal{A}) $
   such that the family
   $\{ b_{ij, N}:\ 1 \leq i \leq j \leq N \} $
   is Boolean independent and
   \begin{enumerate}
   \item[(1)] the  variables
    $ \{ b_{ij, N}: 1 \leq i < j \leq N \} $
    are identically distributed such that
    \begin{enumerate}
    \item[(1.1)] $
     \displaystyle \lim_{N \rightarrow\infty}
     N^\varepsilon \varphi(b_{ij, N}^{\xi_1}b_{ij, N}^{\xi_2}  \cdots b_{ij, N}^{\xi_n} )= 0
     $
     for any $ \varepsilon < 1 $ and
      $ \xi_j \in \{ 1, \ast\} $
      \item[(1.2)] $ \displaystyle
     \lim_{N \rightarrow \infty} N \varphi \big( ( b_{ ij, N} b_{ij, N}^\ast)^m \big) =
 \lim_{N \rightarrow \infty} N \varphi \big( (b_{ij, N}^\ast b_{ ij, N} )^m \big) =0 $
 for all $ m  > 1 $
    \item[(1.3)] $ \displaystyle \lim_{ N \rightarrow \infty} N \varphi\big( b_{ij, N} b_{ij, N}^\ast\big) = \alpha$
    and
$ \displaystyle \lim_{N \rightarrow \infty}
N \varphi\big( b_{ij, N}^\ast b_{ij, N}\big) = \beta $
for some $ \alpha, \beta > 0 $.
    \end{enumerate}
    \item[(2)] the variables $\{ b_{ii, N } : i \in [ N ] \} $
    are self-adjoint, identically distributed such that for any positive integer $n $
    and any $ i \in [ N ] $,
    \[  \lim_{N \rightarrow\infty} \varphi\big( b_{ii, N}^n\big) = 0 .\]
   \end{enumerate}
   Then the asymptotic distribution of $ B_N $  does exist and all its odd moments asymptotically vanish. Moreover, $ B_N $ is asymptotically Bernoulli distributed if and only if $ \alpha = \beta $.
    \end{thm}
 To simplify the writing, we will omit the index $ N $, using
  $ b_{ij} $ for $ b_{ ij, N} $.
\begin{proof}
 The condition that $ B_N  $ is self-adjoint gives that
  $ b_{jk} \in \{ b_{ij}, b_{ij}^\ast \} $
  only if $ k = j $. Thus, denoting
  $ \overrightarrow{\zeta} = (1, \ast, 1, \ast, \dots) $
  we have that
  \begin{align*}
  \varphi \circ \tr \big( B_N^n\big)
   & = \sum_{ \oi \in [ N ]^n }
   \frac{1}{N}
   \varphi( b_{ i_1 i_2} b_{i_2 i_3} \cdots b_{i_n i_1})\\
    &= \sum_{ \sigma \in \mathcal{I}(n)}
  \sum_{ \substack{ \oi \in [ N ]^n\\ \iota( \overrightarrow{\zeta}, \oi) = \sigma}}
  \frac{1}{N} \varphi_{\sigma} [ b_{i_1i_2}, b_{i_2 i_3}, \dots, b_{i_ni_1}]
  \end{align*}
  For
  $ \sigma \in \mathcal{I}(n) $, define
  \begin{align*}
  A(\sigma)&= \{ \oi \in [ N ]^n:\ \iota(\overrightarrow{\zeta}, \oi) = \sigma \text{ and } i_k \neq i_{ k + 1} \text{ for all } k \in [ N ] \}\\
  B ( \sigma, k)  &= \{ \oi \in [ N ]^n:\
 \iota(\overrightarrow{\zeta}, \oi) = \sigma \text{ and } i_k = i_{ k + 1} \}.
  \end{align*}
 First we shall show by induction on $ \# \sigma $ that
 \begin{equation}\label{eq:22}
 \left\{
 \begin{array}{l}
 \displaystyle \lim_{N \rightarrow \infty}
 \sum_{ \oi \in B( \sigma, k)}
 \frac{1}{N} \varphi_{\sigma} [ b_{i_1i_2}, b_{i_2 i_3}, \dots, b_{i_ni_1}]
   =0
 \\
 \displaystyle \limsup_{ N \rightarrow \infty}
 \sum_{ \substack{ \oi \in [ N ]^n\\ \iota( \overrightarrow{\zeta}, \oi) = \sigma}}
   \frac{1}{N}  |\varphi_{\sigma} [ b_{i_1i_2}, b_{i_2 i_3}, \dots, b_{i_ni_1}] |
   < \infty.
 \end{array}
 \right.
 \end{equation}
 Remark that if
 $ i_k = i_{ k + 1} $
 and if
 $ B = \{ l+1, l+2, \dots, l+p \} $
 is the block of $ \sigma $ that contains $k $,
 then
 $ b_{{i_k}, i_{k+1}} = b_{i_{l+1}i_{l+2}} =
 \dots = b_{i_{l+p}i_{l+p+1}} $,
  therefore
 $ i_{l+1} = i_{ l+2} = \dots = i_{l+p +1} $.

 If  $ \sigma $ has a single block, then (\ref{eq:22}) is trivial.
  For the induction step, let us suppose  that
 $ \oi \in B( \sigma, k ) $
  and that
 $ \sigma = \sigma_1 \oplus [ p ] \oplus \sigma_2 $
 with
 $ \sigma_1 \in \mathcal{I}(m) $,
 $ \sigma_2 \in \mathcal{ I }( q) $
 such that
 $ m< k \leq m + p + 1 $
 and
 $ n = m + p + q $
 (i.e. the block containing $ k $ has $ p $ elements).

  Since
  $ i_k = i_{ k+1}$,
  we have that
  $ i_{m+1} = i_{m+1} = \dots = i_{ m +p + 1}$,
  therefore
  \begin{align*}
  \varphi_{\sigma}
  \big[
  b_{i_1 i_2}, \dots, b_{i_n i_1}
  \big] &= \varphi_{ \sigma_1}
  \big[ b_{i_1 i_2}, \dots ,b_{i_{m-1} i_m}
  \big]
  \cdot \varphi\big( b_{i_k i_{k+1}}^p \big)
  \cdot
  \varphi_{\sigma_2}
  \big[
   b_{i_{m+ p + 1} i_{m+p + 2}}, \dots, b_{i_n i_1}\big]\\
   & = \varphi\big( b_{i_k i_{k+1}}^p \big)
   \cdot
   \varphi_{ \sigma_1 \oplus \sigma_2 }
   \big[
   b_{i_1i_2}, \dots ,b_{i_{ m-1} i_{m}},
   b_{ i_m i_{ m+ p + 2}}, \dots, b_{i_ni_1}
   \big]
  \end{align*}
 and (\ref{eq:22}) follows from the induction hypothesis.

Conditions (1.2) and (1.3) give that, for any
   $ k, l \in [ N ] $ such that $ k \neq l $
   \begin{align*}
   \limsup_{N \rightarrow \infty} &
    | N \cdot\varphi\big( b_{kl} b_{kl}^\ast \big)|
   \leq \alpha + \beta \\
   \limsup_{ N \rightarrow \infty} &
    | N \cdot  |\varphi\big( (b_{kl} b_{kl}^\ast)^m \big)|
   =0 \textrm{  if $ m > 1 $.}
   \end{align*}
   Also,
   $ \# A(\sigma) \leq \# \{ \oi\in [ N ]^n: \iota( \overrightarrow{\zeta}, \oi) = \sigma \} $.
 Therefore, if $ \sigma $ is not
 $\overrightarrow{\zeta}$-alternating,
  Lemma \ref{lemma:4} gives
  \[
  \lim_{N \rightarrow\infty}
\sum_{ \substack{ \oi \in [ N ]^n\\ \iota( \overrightarrow{\zeta}, \oi) = \sigma}}
     \frac{1}{N}
 \varphi_{\sigma} [ b_{i_1i_2}, b_{i_2 i_3},
  \dots, b_{i_ni_1}]
  =0,
  \]
  whilst, if
 $ \sigma $ is not
 $ \overrightarrow{\zeta} $-alternating,
 we have that
 \begin{align*}
\limsup_{ N \rightarrow \infty}
 \sum_{ \substack{ \oi \in [ N ]^n\\ \iota( \overrightarrow{\zeta}, \oi) = \sigma}}
   \frac{1}{N} & |\varphi_{\sigma} [ b_{i_1i_2}, b_{i_2 i_3}, \dots, b_{i_ni_1}] |\\
    & =
  \limsup_{ N \rightarrow \infty}
   \sum_{ \oi \in A ( \sigma)}
     \frac{1}{N}  |\varphi_{\sigma} [ b_{i_1i_2}, b_{i_2 i_3}, \dots, b_{i_ni_1}] | \\
     & \leq
     \# A (\sigma) \cdot N^{ \# B } \cdot
     (\alpha + \beta)^{ \#\{ B \in \sigma :\ \#B =2 \} }
       \cdot
       0^{ \#\{ B \in \sigma:\ \#B \neq 2 \}},
 \end{align*}
 hence the proof of (\ref{eq:22}) is complete.

Moreover, the argument above gives that
 \[
  \lim_{N \rightarrow\infty}
\sum_{ \substack{ \oi \in [ N ]^n\\ \iota( \overrightarrow{\zeta}, \oi) = \sigma}}
     \frac{1}{N}
 \varphi_{\sigma} [ b_{i_1i_2}, b_{i_2 i_3},
  \dots, b_{i_ni_1}]
  =0
  \]
unless $ \sigma $ is an interval pair partition. Thus, for $ n $ odd,
\[
\lim_{ N \rightarrow \infty} \varphi \circ \tr \big( B_N^n\big) = 0,
 \]
and, if $ n = 2r $ is even, for
$ \tau = [ 2, 4, \dots, 2r ] $,
\begin{align*}
\lim_{ N \rightarrow \infty}
 \varphi \circ \tr \big( B_N^n\big)
 = \lim_{N \rightarrow \infty}
 \sum_{ \oi \in A(\tau) }
 \frac{1}{N}
 \varphi_{\tau}
 \big[ b_{ i_1 i_2}, \dots, b_{i_n i_1}\big].
\end{align*}
By definition,
$ \oi \in A(\tau) $
if and only if
$ i_{2k-1} = i_1 $
and
 $ i_1 \neq i_{ 2k } \neq i_{ 2k+2} \neq i_1 $
for all $ k \in [r ] $.  In particular,
 \[
 \# A( \tau) = N \cdot ( N -1) \cdot ( N -2)^{ r -1}
 \]
 Thus, denoting $ i_1 = 1 $, $ j_0 = 0 $ and $ i_{2s} = j_s $ for $ s \in [ r ] $,
  we have that
 \begin{align*}
 \sum_{ \oi \in A(\tau) }
  \frac{1}{N}
  \varphi_{\tau}
  \big[ b_{ i_1 i_2}, \dots, b_{i_n i_1}\big]
  &=
  \sum_{ \oi \in A(\tau) }
    \frac{1}{N}[
    \prod_{ s \in [ r ] }
  \varphi \big( b_{i_1 i_{2s}} b_{i_{2s} i_1}^\ast  \big)]\\
 & =
 \frac{1}{N} \sum_{ i =1}^N \prod_{ s =1}^r \big(
 \sum_{ \substack{j_s \in [ N ]\\ i \neq j_s \neq j_{ s-1}}}
 \varphi( b_{ij_s} b_{ij_s}^\ast )
 \big).
 \end{align*}

 For $ i< j $, denote
 $ \alpha _N = \varphi\big( b_{ij, N} b_{ij, N}^\ast\big) $
 and
 $ \beta_N = \varphi \big( b_{ij, N}^\ast b_{ij, N}\big) $.
 In particular
 $ \displaystyle \lim_{N \rightarrow \infty} \alpha_N = \alpha $
 and
 $ \displaystyle \lim_{ N \rightarrow \infty} \beta_N = \beta $.
  Since $ i - 1 $ elements of $ [ N ] $ are strictly less that $ i $ and $ N - i $ are strictly greater than $ i $, we have that
  \[
  \Big(
  \alpha_N \frac{i -1}{N} + \beta_N \frac{N -i -1}{N}
  \Big)^r
  <
    \prod_{ s =1}^r \big(
   \sum_{ \substack{j_s \in [ N ]\\ i \neq j_s \neq j_{ s-1}}}
   \varphi( b_{ij_s} b_{ij_s}^\ast )
   \big)
   <
    \Big(
     \alpha_N \frac{i }{N} + \beta_N \frac{N -i }{N}
     \Big)^r
  \]
  therefore
  \[
  \lim_{N \rightarrow \infty}
  \frac{1}{N} \sum_{ i =1}^N \prod_{ s =1}^r \big(
   \sum_{ \substack{j_s \in [ N ]\\ i \neq j_s \neq j_{ s-1}}}
   \varphi( b_{ij_s} b_{ij_s}^\ast )
   \big)= \int_{0}^1 \big( \alpha x + \beta(1-x)\big)^r dx
 \]
  that is
  \[
  \lim_{N \rightarrow \infty}\varphi \circ \tr\big( B_N^{2r} \big)
     = \left\{
     \begin{array}{l}
     \alpha^r , \text{ if $ \alpha = \beta $}\\
     \displaystyle \frac{\alpha^{r+1} - \beta^{r+1}}{(r+1)( \alpha - \beta)}, \text{ if  } \alpha \neq \beta,
     \end{array}
     \right.
  \]
  hence the conclusion.
\end{proof}


\bibliographystyle{alpha}


\end{document}